\documentclass[11pt,reqno]{amsart} 
\usepackage{amssymb,latexsym} 
\usepackage{graphicx}
\usepackage{fancyhdr,amssymb}
\usepackage{url}
\usepackage{cite}		
\usepackage{enumitem} 
\usepackage{float}
\restylefloat{table}
\usepackage{enumitem}
\usepackage{setspace}

\newcommand{\Z}{{\mathbb Z}}

\theoremstyle{plain}
\numberwithin{equation}{section}
\newtheorem{thm}{Theorem}[section]
\newtheorem{theorem}[thm]{Theorem}

\newtheorem{definition}[thm]{Definition}
\newtheorem{proposition}[thm]{Proposition}

\newtheorem{comment}[thm]{Comment}

\begin{document}
\fancyhead{}
\renewcommand{\headrulewidth}{0pt}
\fancyfoot{}
\fancyfoot[LE,RO]{\medskip \thepage}
\fancyfoot[LO]{\medskip MONTH YEAR}
\fancyfoot[RE]{\medskip VOLUME , NUMBER }

\setcounter{page}{1}

\title[A Method for Uniformly Proving a Family of Identities] 
{A Method for Uniformly Proving a Family of Identities}
\author{Russell Jay Hendel}
\address{Department of Mathematics, Towson University}
\email{rhendel@towson.edu}
\thanks{The author acknowledges the careful reading of
an anonymous referee which significantly contributed to the
organizational clarity of the paper.}
 
\begin{abstract} 
This paper presents both a proof method  and a result. \
The proof method presented
is particularly suitable for uniformly proving families of identities
satisfied by a family of recursive sequences. To illustrate
the method, we study the family of recursive sequences
$F^{(k)}_n = \sum_{i=1}^k F^{(k)}_{n-i}, n \ge 0, k \ge 2,$
with $n$ a parameter varying over integers, and $k$ a parameter
indexing members of the family. 
The main theorem states  
$ F^{(k)}_n = \sum_{j=1}^k P_{k,j} F^{(k)}_{n-jk},$
with $P$ a recursive triangle satisfying  
the triangle recursion $P_{i,j}=2P_{i-1,j}-
P_{i-1,j-1},$ with appropriate initial conditions.
The proof of the theorem  exploits the fact that
characteristic polynomials of identities are divisible by the
characteristic polynomial of the recursion generating the underlying sequence. 
\end{abstract}

\maketitle

\section{Introduction and Main Result}

This paper presents both a proof method and result. The proof
method presented, in contrast to other familiar proof methods
for identities, including the Binet form, generating functions,
and matrix methods, is particularly suitable for proving a family
of identities satisfied by a family of recursive sequences. 

To motivate the method we first state the result proven in
this paper which presents a
family of identities satisfied by  the following family of recursive sequences indexed
by a parameter $k \ge 2.$  
\begin{equation}\label{kacci}
	F^{(k)}_n = \sum_{i=1}^k F^{(k)}_{n-i}.
\end{equation}
with $n$ a parameter varying over all integers.

\begin{comment} Four comments are  now presented
 discussing i) initial values, ii) doubly infinite sequences, 
iii) minimal polynomials, and
iv) the names associated with this family as well as with individual
members.  

In the sequel, we will use the abbreviation
$F^{(k)}$ to refer to the sequence 
$\{F^{(k)}_n\}_{\{\text{all $n$}\}}.$

\textbf{Initial values.} The Online Encyclopedia of Integer Sequences, OEIS,
in presenting the Tribonacci and Tetranacci sequences,
 \cite[A000073]{OEIS},\cite[A000078]{OEIS} , indicates the following
initial values., 
\begin{equation}\label{initialvalues}
 F^{(k)}_i=0, \qquad 0 \le i \le k-2, \qquad F^{(k)}_{k-1}=1, 
\end{equation}
 
However, the literature notes two alternatives to \eqref{initialvalues}:
1) \cite{Noe} presents initial conditions with $k-2$ zeroes followed by 2 ones,
and
2)\cite{Bravo}, while adhering to initial values of $k-1$ zeroes followed by a 1, 
starts the sequence
at $F^{(k)}_{k-2}$ so that $F^{(k)}_1=1.$ In this paper we follow \eqref{initialvalues}
which is the OEIS approach. However, the results of this paper are true for any sequence
satisfying \eqref{kacci} independent of initial conditions.   

\textbf{Doubly infinite sequence.} 
As with any recursive sequence initially defined for non-negative $n$, 
\eqref{kacci} allows extending the sequence to negative indices. 

\textbf{Minimal polynomials.}
For reference and later use, for $k \ge 2,$ we let  $r_k(X)$
indicate the minimal polynomial of $F^{(k)}.$
\begin{equation}\label{rk}
	r_k(X) = X^k - \sum_{i=1}^{k} X^{k-i}.
\end{equation}

The theory associated with 
$r_k(x)$ is well-known: i) $r_k(X)$ is irreducible over $\Z[x],$ and hence 
ii) it is the
\emph{minimal polynomial} for $F^{(k)}.$
iii) Furthermore, the $\Z[x]$-ideal generated by the minimal polynomial contains
the characteristic equation associated with any identity satisfied by
the sequences which the minimal recursion generates. Additionally,
iv) $r_k(X)$ has simple roots; one root is very close to two while the other
roots lie in the unit circle \emph{\cite{Miles, Miller, Wolfram2}}.
 However, in the sequel we will not need  the entire theory.

\textbf{Names.} For $k=2,3,4$ and for traditional initial values,
\eqref{initialvalues},
\eqref{kacci} corresponds to the 
Fibonacci, 
Tribonacci \emph{\cite[A000073]{OEIS}}, and Tetranacci \emph{\cite[A000078]{OEIS}} sequences respectively.

However, there is no accepted standardized name for the family of sequences 
described by \eqref{kacci}. Furthermore, when names do exist, 
they are also used to refer to other types of families of identities. 
Consequently, in this paper we will suffice with using the algebraic notation,
$F^{(k)}.$ Here are some details supporting this assertion.

An initial place to look for names for the family 
described by \eqref{kacci} would be the Fibonacci Quarterly. 
But we find multiple terms such as 
\emph{generalized Fibonacci numbers} \cite{Ferguson, Kalman}, 
\emph{$k$-generalized Fibonacci numbers} \cite{Flores},  and
\emph{generalized Fibonacci $k$-sequences} \cite{Gabai}.
Dresden \cite{Dresden} lists five accepted names for the family 
described by \eqref{kacci}: 
\emph{i) generalized Fibonacci numbers of order $k,$ 
ii) Fibonacci numbers of order $k$, 
iii) Fibonacci $k$-step numbers, 
iv) Fibonacci $k$-sequences, and 
v) $k$-bonacci numbers.} If one looks for rarer names 
occurring in only one or two articles other names pop up 
such as \emph{the $k$-nacci numbers} \cite{Bacani}. But even here,
a routine google search over the internet using the search
term ``$k$-nacci" shows multiple papers.

It might appear from the preceding discussion that the 
description
\emph{generalized Fibonacci numbers} is a safe consensus for a name. 
But \cite{Edson, Yayenie}, 
citing multiple references,  point out that in 
generalizing the Fibonacci recursion, $F_n=a F_{n-1} + bF_{n-2}, a=b=1,$ 
some authors   have altered the starting values, 
while others  have preserved the first two terms of the sequence 
but changed the recurrence relation. 
In other words, the phrase \emph{generalized Fibonacci numbers} 
can refer to different things in different articles.

As indicated at the introduction of this comment, 
we find the best course to simply modestly state that we are 
proving a family of identities satisfied by a family of recursive sequences 
$F^{(k)},$ which for $k=2,3,4$ are the familiar Fibonacci, Tribonacci, and
Tetranacci sequences (with appropriate initial values). 
\end{comment}

We motivate the main result by considering three examples.
\begin{itemize}
\item
The subsequence of Fibonacci numbers restricted to even indices,
$\{G_n\}_{n \ge 0}=$ \\ 
$\{0,1,3,8,21,55,\dotsc\},$ satisfies the recursion 
$G_n=3G_{n-1}-G_{n-2}.$ Equivalently, $\{F_n\}_{n \ge 0}$ satisfies
$F_n = 3F_{n-2} - F_{n-4}.$ Note that this last identity is satisfied by the 
Fibonacci numbers restricted to either  even or odd indices. 
\item Similarly, the subsequence of Tribonacci
numbers, as defined by the OEIS, 
\cite[A000073]{OEIS}, restricted to indices divisible by 3,
$\{G_n\}_{n \ge 0}=\{0, 1,  7,  44,  274,  1705,  10609, \dotsc \},$
satisfies the recursion $G_n = 7 G_{n-1} -5 G_{n-2} + G_{n-3}.$
For purposes of this paper, we will deal with the corresponding identity
on the Tribonacci numbers, $F_n^{(3)} = 7 F_{n-3}^{(3)} - 5 F_{n-6}^{(3)} 
+ F_{n-9}^{(3)}.$ As in the
case of the Fibonacci numbers, it is simply a convenience to say that the identity
holds for the subsequence of indices divisible by three; it actually also holds for
any subsequence of indices at a fixed congruence modulo 3. 
\item The subsequence of the Tetranacci numbers,
as defined by the OEIS, \cite[A000078]{OEIS},
 restricted to indices divisible by 4,
$\{G_n\}_{n \ge 0} =$ 
$ \{0,  1,  15,  208,  2872,  39648,  547337, \dotsc \},$ satisfies the recursion,
$G_n = 15G_{n-1} - 17 G_{n-2} + 7 G_{n-3} -G_{n-4},$ or equivalently, the 
Tetranacci numbers satisfy the recursion 
$F_n^{(4)} = 15F_{n-4}^{(4)} - 17 F_{n-8}^{(4)} + 7 F_{n-12}^{(4)} -F_{n-16}^{(4)}.$
 \end{itemize}

These examples motivate defining a triangle $P,$ presented in Table \ref{TP},
whose second, third, and fourth 
rows (starting the count of rows from the zero-th row)
 correspond to the coefficients in the recursions just listed.

\begin{center}
\begin{table}[H]
\begin{tabular}{||c|c c c c c c c||}
\hline \hline \text{Row number-Column Number} &0 &  1 & 2 & 3 & 4 & 5 & 6 \\
\hline
0				  & -1 & 0 & 0 & 0 & 0 & 0 & 0 \\
1                                 & -1 & 1 & 0 & 0 & 0 & 0 & 0   \\
2                                 & -1 & 3 & -1& 0 & 0 & 0 & 0   \\
3                                 & -1 & 7 & -5& 1 & 0 & 0 & 0   \\
4                                 & -1 &15 & -17&7 &-1 & 0 & 0   \\
5				  & -1 & 31 & -49 &31&-9 & 1&  0   \\
6				  & -1 & 63 &-129&111&-49&11& -1   \\
\hline \hline
\end{tabular}
\caption{The recursive triangle $P.$} 
\label{TP}
\end{table}
\end{center} 

To formulate our main result we instead define $P$ recursively as follows.

\begin{definition} \; \\
\begin{equation}\label{2minus1}
P_{i,j} = 2P_{i-1,j} - P_{i-1,j-1}. 
\end{equation}
\end{definition} 

The initial values for $P$ are as given in Table \ref{TP} in row 0 and column 0.
As with any triangle recursion, once initial values are given, the associated
sequence can be extended to a doubly infinite array. This fact will be
useful in the proof.

Having defined $P,$ we can then formulate a   main theorem 
which states that the rows of $P$ give the coefficients
of the recursion satisfied by the subsequence of 
$\{F^{(k)}_n\}_{\{\text{all n}\}}$  restricted
to indices divisible by $k.$

\begin{theorem} For  $k\ge 2,$ (and all $n$) 
\begin{equation}\label{t1}
	F^{(k)}_{n} = \sum_{i=1}^k P_{k,i} F^{(k)}_{n-ki}.
\end{equation}
\end{theorem}

\section{Proof Method}

This section motivates the need for a separate proof method.
Traditional  proof methods include the Binet form, generating
functions, and matrix methods \cite{Koshy}. If we are proving
an identity satisfied by a single recursive sequence, these methods suffice.

However, we are proving that for each $k,$   \eqref{t1} is 
satisfied by $F^{(k)}.$ We therefore need a \emph{uniform} method to
prove the family of identities. The Binet form for each $F^{(k)}$
is different with different roots. Similar comments apply to generating functions and matrices. 

The method by which we prove \eqref{t1} exploits the fact 
that the minimal polynomial of 
a recursive sequence generates an ideal containing 
the characteristic polynomials corresponding to 
all identities satisfied by the recursive sequence.
This method will allow uniform proofs.

To make this precise, we first illustrate the flow of logic 
in the proof, by discussing the proof  for $k=2,$
 that the  
subsequence of Fibonacci numbers restricted to even indices,
$\{G_n\}_{n \ge 0}=$   
$\{0,1,3,8,21,55,\dotsc\},$ satisfies the recursion 
\begin{equation}\label{tt1}
	G_n=3G_{n-1}-G_{n-2}.
\end{equation}.

As just formulated this appears to be statement about a sequence
$\{G_n\}_{n \ge 0}.$ We can equivalently prove,
\begin{equation}\label{tt2}
	F_n = 3 F_{n-2} - F_{n-4},
\end{equation}  which is a statement about the Fibonacci numbers.
We simply observe that \eqref{tt2} implies \eqref{tt1}.

\begin{definition} If $H_m = \sum_{i=1}^n c_i H_{m-i}$ is an identity
satisfied by a recursive sequence $\{H_m\}_{m \ge 0},$ then the 
 characteristic polynomial of the identity is
\begin{equation}\label{cp}
	p(X) = X^n - \sum_{i=1}^n c_i X^{n-i}.
\end{equation} 
\end{definition}

Equation \eqref{cp}
 establishes a correspondence between recursive sequences and polynomials.

This definition of characteristic polynomial follows
\cite[Section 6.2]{Rosen}. As noted in \cite{Hendel} other definitions
do exist. Definition 2.1 regards the characteristic polynomial
as an attribute of the identity. Contrastively, the minimal
polynomial is an attribute of the recursive sequence; roughly, it is
the characteristic polynomial with leading coefficient  1 
of smallest degree whose associated recursion is satisfied by the sequence.
The following result is well-known, but  is not frequently used in traditional sources
\cite{Koshy} to prove identities. Webb \cite{Webb1,Webb2} has written
some papers showing the usefulness of such methods. The proof is 
short and also given below.

\begin{proposition} A recursive identity is satisfied by a recursive sequence iff the characteristic polynomial of that
identity is divisible by the minimal polynomial of the underlying
recursion. \end{proposition}

\begin{proof} Interpret the variable $X$ as the difference operator 
\cite{Boole}. Then the minimal polynomial considered as an operator when applied to the
sequence to which it is a minimal polynomial annihilates that sequence, that is,
the resulting sequence is identically 0. 
It immediately follows by the associativity of polynomials 
that any polynomial multiple of this minimal polynomial 
also annihilates the underlying recursive sequence showing 
that the identity corresponding to this multiple of the minimal polynomial 
is satisfied by the recursive sequence. 
Since in $\Z[x]$ all ideals are generated by a single element, 
it follows that the characteristic polynomial associated with any 
identity must lie in the ideal generated by the minimal polynomial
\end{proof}

It follows that to prove \eqref{tt2} it suffices to show that 
the characteristic polynomial of this identity 
is divisible by the minimal polynomial for the Fibonacci numbers, 
$r(X)= X^2-X-1.$ 

The flow of logic is concisely presented in Table \ref{Tpolynomial}.

\begin{center}
\begin{table}[H]
\begin{tiny}
\begin{tabular}{||p{.15\linewidth} | p{.35\linewidth} || p{.30\linewidth} | p{.20\linewidth}||}
\hline \hline Row number & Item &  Recursive-Identity Formulation &  Characteristic-Polynomial Formulation \\
\hline
1 & Fibonacci Numbers, & $F_n=F_{n-1}+F_{n-2}$ 	& $r(X)=X^2 - X - 1$ \\
2 & Even-indexed Fibonacci numbers  & $F_n = 3F_{n-2} -F_{n-4}$	& $p(X)=X^4 -3X^2 +1$  \\
\hline
3 & The Row-2 recursion is satisfied by the Fibonacci numbers   & &
$p(X)$ is divisible by $r(X)$ \; \\
\hline \hline
\end{tabular}
\caption{ Polynomial formulation of Theorem. } 
\label{Tpolynomial}
\end{tiny}
\end{table}
\end{center} 
To prove that $p(X)$ is divisible by $r(X)$ 
it suffices to show that the quotient $p(X)/r(X)=q(X)$ 
lies in $\Z[x]$ that is, has integer coefficients. 

For the case of general $k,$  
first, we define $p_k(X), k \ge 2,$ as follows.
\begin{equation}\label{pk}
		p_k(X) = X^k - \sum_{i=1}^k P_{k,i} X^{k-i}.
\end{equation}
Then the characteristic polynomial of \eqref{t1} is $p_k(X^k).$
$p_k(X)$ has descending powers of $X$ while $p_k(X^k)$ 
has powers which descend by multiples of $k$ as required. 

By the discussion above, 
to prove \eqref{t1} it suffices to prove that
the \emph{corresponding} characteristic polynomial 
$p_k(X^k)$ is divisible by the minimal polynomial of $F^{(k)},
$ $r_k(X).$ Alternatively, we must show
that $\frac{p_k(X^k)}{r_k(X)}$ lies in $\Z[x].$

We therefore proceed as follows.
\begin{proposition} Define rational functions, $q_k(x), k \ge 2$ by
\begin{equation}\label{qk}
	p_k(X^k) = r_k(X) q_k(X).
\end{equation} 
Then to prove \eqref{t1}, it suffices to prove that $q_k(X) \in \Z[X].$
\end{proposition}

Actually, in this paper, we prove the equivalent equation,
\begin{equation}\label{qkn}
	-p_k(X^k) = (-r_k(X)) q_k(X).
\end{equation}
This is purely a convenience to avoid excessive minus signs in tables and equations.

To summarize, we will prove \eqref{qkn} which in turn implies
\eqref{qk} which in turn shows $\frac{p_k(X^k)}{r_k(X)}$ lies in $\Z[x],$
 or equivalently that $p_k(X^k)$ is divisible by 
$r_k(X).$ By Proposition 2.2 this implies that 
\eqref{t1}, the identity corresponding to the characteristic polynomial $p_k(X^k),$
is satisfied by $F^{(k)}$ the recursive sequence corresponding 
to the minimal polynomial $r_k(X).$ In other words, proving 
\eqref{qkn} suffices to prove Theorem 1.3.

We close this section explaining where the \emph{work} for the proof lies. 
We must show \emph{for each $k$} that $q_k(X)$ has integral coefficients.  
What we actually do, is inductively describe the coefficients of $q_k(X)$ 
and then show that 
$q_k(X)(-r_k(X))=-p_k(X^k).$

\section{The Coefficients of $q_k(X).$}

 To describe the  coefficients
of $q_k(x)$ we first define a matrix $Q,$ presented in Table \ref{TQ}, which
  is obtained by applying the backward difference operator to $P.$  

\begin{definition} The triangle $Q$ is defined by
\begin{equation}\label{Q}
	Q_{i,j} = P_{i,j}-P_{i-1,j}.		
\end{equation}
\end{definition} 
\begin{center}
\begin{table}[H]
\begin{tabular}{||c|c c c c c c c||}
\hline \hline
$\text{Row number-Column Number}$ & 0	& 1 & 2 & 3 & 4 & 5 & 6 \\
\hline
0				  & 0 & 0 & 0 & 0 & 0 & 0 & 0 \\	
1                                 & 0 & 1 & 0 & 0 & 0 & 0 & 0 \\
2                                 & 0 & 2 & -1& 0 & 0 & 0 & 0 \\
3                                 & 0 & 4 & -4& 1 & 0 & 0 & 0 \\
4                                 & 0 & 8 & -12&6 &-1 & 0 & 0 \\
5		   		  & 0 & 16& -32 &24&-8 & 1&  0 \\
6	                   	  & 0 & 32&-80&80&-40&10& -1 \\
\hline \hline
\end{tabular}
\caption{ The recursive triangle $Q$ defined by Definition \eqref{Q}. }  
\label{TQ}
\end{table}
\end{center} 

In the proof of Theorem 1.3, besides using  \eqref{Q}, we will need
the following two easily proven identities. These three identities allow
us to convert sums in $Q$-elements into $P$-elements.
\begin{subequations}\label{Q2P}
\begin{align}
 \; &		  \sum_{r=c}^t Q_{r,c} = P_{t,c}, &\text{by telescoping and \eqref{Q}},\\
 \; &		 \sum_{r=s}^t Q_{r,c}= P_{t,c} - P_{s-1,c}. 
		 &\text{ by the equation just stated.}     
\end{align}
\end{subequations}


Using \eqref{Q}  we can now completely describe 
the polynomial $q_k(x),$  $k \ge 2.$
\begin{equation}\label{qq}
	q_k(x) = x^{k(k-1)} + 
	\displaystyle \sum_{j=0}^{k-2} x^{kj} 
	\Biggl(-P_{k-1,k-1-j} +
	\displaystyle \sum_{i=k-1-j}^{k-1} Q_{i,k-1-j} x^{k-i} \Biggr).
\end{equation}

As explained in the last section, to prove \eqref{t1} it suffices to show
\eqref{qkn} is true which is what we will spend the rest of the paper doing.

\section{Some Literature}

Before continuing the proof, we take note of certain related items
appearing in the literature.  

\cite[A193845]{OEIS}  presents a triangle similar to $P$ 
with only positive entries. This triangle can be generated   by power series,
polynomials,  or 
using fission and fusion \cite{Kimberling}. 
It is easy to "fix" this OEIS triangle
so that it agrees with $P$ in sign. 

For example, M. F. Hasler in the last comment of the \emph{comment} section
of sequence \cite[A193845]{OEIS}, 
dated Oct. 15, 2014, 
 notes that the $k$-th row of A193845 lists the coefficients of the polynomial 
$\sum_{i=0}^k (X+2)^i,$ in order of increasing powers. 

To modify this for the $P$ triangle in Table \ref{TP}, 
which has negative entries,  we first replace $X+2$ with $2-X.$ 
An example, for $k=3,$ will illustrate the further 
required polynomial manipulations. 
We can verify that 
$(2-X)^3 - (2-X)^2-(2-X)-(2-X)^0 = -X^3+5X^2-7X+1$ which is almost what we want except
that the signs and exponents have to be adjusted. 
First, by \eqref{rk} we can re-write this as 
$r_3(2-X) = (2-X)^3 - (2-X)^2-(2-X)-(2-X)^0 =-X^3+5X^2-7X+1.$ 
To fix up the sign, we 
can, by \eqref{pk}, write $X^3 r_3(2-\frac{1}{X}) = 
-1+ 5X-7X^2+X^3 = p_3(X).$ 
We can similarly write for arbitrary $k,$ 
$X^{k^2} r_k(2-\frac{1}{X^k}) = p_k(X^k).$ 

This gives an alternate definition of triangle $P$ from which the recursion 
\eqref{2minus1} can
be derived. Expositionally, since all we need is the underlying recursion, 
it is simpler to
directly define triangle $P$ by \eqref{2minus1}. 

Similar observations can be made about triangle $Q$ in Table \ref{TQ}. 
A  triangle similar to $Q$ with only positive entries 
is found in \cite[A038207]{OEIS}. N-E. Fahssi, in the \emph{comment} section
of this sequence (with comment date, Apr 13 2008),  
points out that the $k$-th row of A038207 
are the coefficients $(2+X)^k$.

We can adjust this idea to triangle $Q$ 
which has negative entries. The coefficients of $X(2-X)^k$ give 
the coefficients in the $k$-th row of triangle $Q.$
For example, when $k=2,$ $X(2-X)^2 =4X-4X^2+X^3$ 
corresponding to row 3 in 
Table \ref{TQ}.
     
These alternate definitions might provide an
alternate proof of \eqref{t1} using polylnomials. 
A challenge in proving \eqref{qkn} by polynomials or some similar
method such as fusion and fission is that, as seen in \eqref{qq}, the 
coefficients of $q_k(X)$ are mixtures of $Q$ and $P$ triangle entries. 
The proof presented in this paper is both straightforward and computational, and should
have independent interest.
 We therefore
leave the search for a purely polynomial proof of \eqref{qkn} as an open problem.

 \section{ The matrix $M$}

We prove \eqref{qkn} computationally by performing the product long hand. To keep tabs on
calculations we store the information in a matrix $M=M_k.$ Throughout the proof,
Tables \ref{TrqQP14}-\ref{Trqnumber25} presenting $M_5$ facilitate following arguments.

The matrix $M,$ as illustrated in Tables  \ref{TrqQP14} and \ref{TrqQP25} for the 
case $k=5,$
labels columns with $X$ raised to a power. In the sequel we 
will interchangeably refer to the \emph{column with header $X^k$} 
as column $X^k$ or column $k.$ 

The rows of $M$ are labeled with the summands in $-r_k(X),$ \eqref{rk}. Similar,
to the way we refer to columns, in the sequel we will refer, for example, to the
row with header $-X^k$ as either row $k$ or row $X^k$ (leaving out the sign).
This should cause no confusion. 

Using this terminology, we define $M$ as follows:
\begin{small}
\begin{subequations}\label{M}
\begin{align}
 \; & M_{r,c}&= \text{coefficient of $X^c$ in $X^r q_k(X)$}, & 0 \le r \le k-1\\
 \; & \;    &= \text{coefficient of $X^{c-r}$ in $q_k(X)$}, & c \ge r, r \neq k\\
 \; & M_{k,c}&= \text{coefficient of $X^c$ in $-X^k q_k(X)$}, &\; \\
 \; & \;     &= \text{coefficient of $X^{c-k}$ in $-q_k(X)$}, & c \ge k
\end{align}
\end{subequations}
\end{small}
	
 For   $k=5,$
Tables \ref{TrqQP14} and \ref{TrqQP25} present all non-leading
coefficients of $q_k(X)$ 
using $Q$ and $P$ with subscripts, while Tables \ref{Trqnumber14} and \ref{Trqnumber25}  
present the 
associated numerical values. The leading coefficient for the $X^{20}$ column
is simply 1 in both Tables \ref{TrqQP25} and \ref{Trqnumber25}.  Tables \ref{TrqQP14}-\ref{Trqnumber25} have a terminal 
\emph{sum} row indicating in each column the sum of the entries
above it in that column. The top \emph{case}
row is explained in the next section.

\begin{center}
\begin{table}[H]
\begin{tiny}
{\renewcommand{\arraystretch}{1.3}
\begin{tabular}{|||c||c|c|c|c|c||c|c|c|c|c||c|c|c|c|c||}
\hline Case& A&B&B&B&B&C&E&B&B&B&C&E&E&B&B \\
\hline$r-c$&$X^{0}$&$X^{1}$&$X^{2}$&$X^{3}$&$X^{4}$&$X^{5}$&$X^{6}$&$X^{7}$&$X^{8}$&$X^{9}$&$X^{10}$&$X^{11}$&$X^{12}$&$X^{13}$&$X^{14}$\\
\hline$X^0$&$-P_{4,4} $&$Q_{4,4}$&\;&\;&\;&$-P_{4,3} $&$Q_{4,3}$&$Q_{3,3}$&\;&\;&$-P_{4,2} $&$Q_{4,2}$&$Q_{3,2}$&$Q_{2,2}$&$$\\
\hline$X^1$&$ $&$-P_{4,4} $&$Q_{4,4}$&$ $&$ $&$ $&$-P_{4,3} $&$Q_{4,3}$&$Q_{3,3}$&$ $&$ $&$-P_{4,2} $&$Q_{4,2}$&$Q_{3,2}$&$Q_{2,2}$\\
\hline$X^2$&$ $&\;&$-P_{4,4} $&$Q_{4,4}$&$ $&$ $&$ $&$-P_{4,3} $&$Q_{4,3}$&$Q_{3,3}$&$ $&$ $&$-P_{4,2} $&$Q_{4,2}$&$Q_{3,2}$\\
\hline$X^3$&\;&\;&\;&$-P_{4,4} $&$Q_{4,4}$&$ $&$ $&$ $&$-P_{4,3} $&$Q_{4,3}$&$Q_{3,3}$&$ $&$ $&$-P_{4,2} $&$Q_{4,2}$\\
\hline$X^4$&\;&\;&\;&\;&$-P_{4,4} $&$Q_{4,4}$&$ $&$ $&$ $&$-P_{4,3} $&$Q_{4,3}$&$Q_{3,3}$&$ $&$ $&$-P_{4,2} $\\
\hline$-X^5$&\;&\;&\;&\;&$ $&$P_{4,4}$&$-Q_{4,4}$&\;&\;&\;&$P_{4,3}$&$-Q_{4,3}$&$-Q_{3,3}$&\;&$$\\
\hline$SUM$&$P_{5,5}$&\;&\;&\;&\;&$P_{5,4}$&\;&\;&\;&\;&$P_{5,3}$&\;&\;&\;&$$\\
\hline \hline \hline
\end{tabular}}
\end{tiny} 
\caption{Coefficients of of $-r_5(X) \times q_5(X)$ using $Q,P,$ columns 0-14. } 
\label{TrqQP14} 
\end{table}
\end{center}

\begin{center}
\begin{table}[H]
\begin{tiny}
{\renewcommand{\arraystretch}{1.4}
\begin{tabular}{|||c||c|c|c|c|c||c|c|c|c|c||c|c|c|c|c||}
\hline \hline \hline
\hline Case&C&E&E&E&B&D&D&D&D&D&A\\
\hline$row - col$&$X^{15}$&$X^{16}$&$X^{17}$&$X^{18}$&$X^{19}$&$X^{20}$&$X^{21}$&$X^{22}$&$X^{23}$&$X^{24}$&$X^{25}$\\
\hline$X^0$&$-P_{4,1} $&$Q_{4,1}$&$Q_{3,1}$&$Q_{2,1}$&$Q_{1,1}$&$1$&\;&\;&\;&\;&$$\\
\hline$X^1$&$ $&$-P_{4,1} $&$Q_{4,1}$&$Q_{3,1}$&$Q_{2,1}$&$Q_{1,1}$&$1$&$ $&$ $&$ $&$ $\\
\hline$X^2$&$Q_{2,2}$&$ $&$-P_{4,1} $&$Q_{4,1}$&$Q_{3,1}$&$Q_{2,1}$&$Q_{1,1}$&$1$&$ $&$ $&$ $\\
\hline$X^3$&$Q_{3,2}$&$Q_{2,2}$&$ $&$-P_{4,1} $&$Q_{4,1}$&$Q_{3,1}$&$Q_{2,1}$&$Q_{1,1}$&$1$&$ $&$ $\\
\hline$X^4$&$Q_{4,2}$&$Q_{3,2}$&$Q_{2,2}$&$ $&$-P_{4,1} $&$Q_{4,1}$&$Q_{3,1}$&$Q_{2,1}$&$Q_{1,1}$&$1$&$ $\\
\hline$-X^5$&$P_{4,2}$&$-Q_{4,2}$&$-Q_{3,2}$&$-Q_{2,2}$&\;&$P_{4,1}$&$-Q_{4,1}$&$-Q_{3,1}$&$-Q_{2,1}$&$-Q_{1,1}$&$-1$\\
\hline$SUM$&$P_{5,2}$&\;&\;&\;&\;&$P_{5,1}$&\;&\;&\;&\;&$-1$\\
\hline \hline \hline
\end{tabular}}
\end{tiny}
\caption{ Coefficients of $-r_5(X) \times q_5(X)$ using $Q,P,$ columns 15-25. } 
\label{TrqQP25} 
\end{table}
\end{center}

\begin{center}
\begin{table}[H]
\begin{tiny}
{\renewcommand{\arraystretch}{1.4}
\begin{tabular}{|||c||c|c|c|c|c||c|c|c|c|c||c|c|c|c|c||}
\hline \hline
\hline$row - col$&$X^{0}$&$X^{1}$&$X^{2}$&$X^{3}$&$X^{4}$&$X^{5}$&$X^{6}$&$X^{7}$&$X^{8}$&$X^{9}$&$X^{10}$&$X^{11}$&$X^{12}$&$X^{13}$&$X^{14}$\\
\hline$X^0$&$1$&$-1$&\;&\;&\;&$-7$&$6$&$1$&\;&\;&$17$&$-12$&$-4$&$-1$&$$\\
\hline$X^1$&$ $&$1$&$-1$&\;&\;&\;&$-7$&$6$&$1$&$\;$&$\;$&$17$&$-12$&$-4$&$-1$\\
\hline$X^2$&\;&\;&$1$&$-1$&\;&\;&\;&$-7$&$6$&$1$&\;&\;&$17$&$-12$&$-4$\\
\hline$X^3$&\;&\;&\;&$1$&$-1$&\;&\;&\;&$-7$&$6$&$1$&\;&\;&$17$&$-12$\\
\hline$X^4$&\;&\;&\;&\;&$1$&$-1$&\;&\;&\;&$-7$&$6$&$1$&\;&\;&$17$\\
\hline$-X^5$&\;&\;&\;&\;&\;&$-1$&$1$&\;&\;&\;&$7$&$-6$&$-1$&\;&$$\\
\hline$SUM$&$1$&\;&\;&\;&\;&$-9$&\;&\;&\;&\;&$31$&\;&\;&\;&$$\\
\hline \hline \hline
\end{tabular}}
\end{tiny}
\caption{Numerical coefficients of $-r_5(X) \times q_5(X),$ columns 0-14. } 
\label{Trqnumber14}  
\end{table}
\end{center}

\begin{center}
\begin{table}[H]
\begin{tiny}
{\renewcommand{\arraystretch}{1.4}
\begin{tabular}{|||c||c|c|c|c|c||c|c|c|c|c||c|c|c|c|c||}
\hline \hline \hline
\hline$row - col$&$X^{15}$&$X^{16}$&$X^{17}$&$X^{18}$&$X^{19}$&$X^{20}$&$X^{21}$&$X^{22}$&$X^{23}$&$X^{24}$&$X^{25}$\\
\hline$X^0$&$-15$&$8$&$4$&$2$&$1$&$1$&\;&\;&\;&\;&$$\\
\hline$X^1$&$ $&$-15$&$8$&$4$&$2$&$1$&$1$&\;&\;&\;&$$\\
\hline$X^2$&$-1$&$ $&$-15$&$8$&$4$&$2$&$1$&$1$&\;&\;&$$\\
\hline$X^3$&$-4$&$-1$&$ $&$-15$&$8$&$4$&$2$&$1$&$1$&\;&$$\\
\hline$X^4$&$-12$&$-4$&$-1$&$ $&$-15$&$8$&$4$&$2$&$1$&$1$&$$\\
\hline$-X^5$&$-17$&$12$&$4$&$1$&$ $&$15$&$-8$&$-4$&$-2$&$-1$&$-1$\\
\hline$SUM$&$-49$&\;&\;&\;&\;&$31$&\;&\;&\;&\;&$-1$\\
\hline \hline \hline
\end{tabular}}
\end{tiny}
\caption{ Numerical coefficients of $-r_5(X) \times q_5(X),$ columns 15-25. } 
\label{Trqnumber25}  
\end{table}
\end{center}

As can be seen from the $X^0$ row of Tables \ref{Trqnumber14}-\ref{Trqnumber25},
\begin{multline}\label{qk5}
  q_5(X) = X^{20} + \biggl( X^{19}
+2X^{18}+4X^{17}
+8X^{16} -15 X^{15} \biggr) \\
+ \biggl( -X^{13} -4 X^{12} -12 X^{11} + 17 X^{10}\biggr)
+ \biggl(  X^7 + 6X^6- 7X^5 \biggr) + \biggl(-X+1\biggr).
\end{multline}

Inspecting  the expression  
inside the summation sign of \eqref{qq},  for any $j, 0 \le j \le k-2,$
\begin{equation}\label{expression}
	X^{kj} 
	\Biggl(-P_{k-1,k-1-j} +
	\displaystyle \sum_{i=k-1-j}^{k-1} Q_{i,k-1-j} X^{k-i} \Biggr),
\end{equation}
we note certain patterns in each such block.  
More specifically, the $j$-th block (from left to right) starts with a 
$P$ term, followed by $Q$ terms, followed by 0 terms.
In \eqref{qk5}, the $j$-th block, $0 \le k \le 3,$ 
 is indicated by the use of bigger parentheses.

Formally,  for each $j, 0 \le j \le k-2,$
\begin{itemize}
 \item $M_{0,jk}=-P_{k-1,k-1-j},$ 
\item $M_{0,jk+m}=Q_{k-m,k-1-j},$ for $1 \le m \le j+1,$ 
\item $M_{0,jk+m}=0,$ for $j+2 \le m \le k-1.$
\end{itemize}

Additionally, geometrically, each successive row after the first row shifts the row above
it by one column to the right as seen in Tables \ref{TrqQP14} and \ref{TrqQP25}
(with the last row having negative signs).
 This shifting is formally 
justified by \eqref{M}. In the sequel we will simply refer
to \eqref{M} for justifications, it being understood that \eqref{M} is geometrically
clarified by these bullets.

\section{Overview of Proof}

To prove \eqref{qkn} using the matrix $M,$ we must show 
that for each column of $M,$ the sum of coefficients in rows 0 through k
equals the corresponding coefficient in $-p_k(X^k),$ found in the bottom \emph{sum} row.

That is, we must show
\begin{subequations}\label{SumMrc}
\begin{align}
\;	&	\sum_r M_{r,jk} &=  P_{k,j}, &\qquad & 0 \le j \le k, \\  
\;	& 	\sum_r M_{r,jk+m} &= 0,  & \qquad &  0 \le  j \le k-1, 1 \le m \le k-1. 
\end{align}
\end{subequations}

Thus to prove \eqref{qkn} we must prove \eqref{SumMrc}. We prove \eqref{SumMrc}
separately on the following five groups of columns. The proof of \eqref{SumMrc}
for each group of columns will occupy one of the next five sections.
\begin{enumerate}[label=\Alph*.]
\item     Columns $k^2$ and 0.
\item  Columns $jk+m, 0 \le j \le k-2, j+1 \le m \le k-1 $
\item   Columns  $jk, 1 \le j \le k-2 $ 
\item Columns $jk+m, j=k-1, 0 \le m \le k-1$
\item  Columns $jk+m,$ for 
\begin{equation}\label{jkm}
	0 \le j \le k-2, 1 \le m \le j 
\end{equation}
\end{enumerate}

For the reader's convenience, the letters denoting cases are
included in Tables \ref{TrqQP14}-\ref{TrqQP25}.
As shown, these five cases are mutually exclusive and 
exhaust all columns.
Throughout the proof of these five cases, 
presented in the next five sections,
we i) review 
Tables \ref{TrqQP14} and \ref{TrqQP25}   in order
to discover patterns, and, using \eqref{M}, ii) state the
pattern discovered for $k=5,$ for general $k.$

\section{Proof of \eqref{SumMrc} for Case A, Columns 0 and $k^2$}
By \eqref{2minus1} and the initial conditions for triangle $P$
presented in Table \ref{TP}, we have for $k \ge 1,$ 
\begin{equation*}
	-P_{k-1,k-1} = P_{k,k}.
\end{equation*}
It follows that for column $0$ the \emph{sum} row
equals the sum of entries above it, verifying \eqref{SumMrc} 
for column $X^{0}.$ For the column $X^{k^2},$ \eqref{SumMrc} reduces
to verification that $-X^k \times X^{k(k-1)} = -X^{k^2}.$

\section{Proof of \eqref{SumMrc} for Case B, columns $jk+m, 
0 \le j \le k-2, j+1 \le m \le k-1 $}

For the case $k=5,$ 
Table \ref{TrqQP14} shows that 
the three elements of block 1 at columns 5,6, and 7 of row $X^0,$
diagonally descend so that these three elements comprise the entries in columns 7,8 and 9.
Similarly, the   four elements in block 2 at columns 10,11,12,and 13 in row $X^0$
diagonally descend so that these four elements comprise the entries in columns 13 and 14.

This pattern is true for general $k.$ By \eqref{M},  
 for the $j$-th block, $1 \le j \le k-2,$
the set of column elements for columns
$jk+m, 0 \le j \le k-2, j+1 \le m \le k-1 $ is equal to the
set of entries in the 0-th row of that block. 

Hence, for $0 \le j \le k-2, j+1 \le m \le k-1,$   
\begin{equation*} 
\left.
\begin{aligned}  
	\;& \text{\; the sum of entries in the $jk+m$-th column,} & \;    \\
	 &=\sum_{r=0}^{k} M_{r,jk+m}   & \; \\
	&=-P_{k-1,k-1-j} + \sum_{r=k-1-j}^{k-1} Q_{r,k-1-j} =0, & \text{by \eqref{Q2P}(a)}.  
\end{aligned}
\right\}
\end{equation*}

Notice that by \eqref{2minus1}, \eqref{Q}, and the initial conditions, row 0 and column 0
of Table \ref{TP}, $Q_{r,c} = 0$ if $r<c.$ These 0 entries correspond to blank cells
in Tables \ref{TrqQP14}-\ref{Trqnumber25}.

\section{Proof of \eqref{SumMrc} for Case C,  columns  $jk, 1 \le j \le k-2 $ }

Table \ref{TrqQP14} shows that the three elements in the 
0-th row of  block 1 at columns 5,6, and 7 
diagonally descend to column 10, the beginning column of the  2nd block,   except that
i) the $-P$ term from 
column 5, now has a positive sign and ii) column 10
 has an additional $-P$ term in row $X^0.$ 

By \eqref{M}, for general $k,$  column $jk, 1 \le j \le k-2,$
has  i) all $Q$ terms
from the 0-th row of the $(j-1)$-st block, ii) a $P$
term from column $(j-1)k$ with a positive sign, and 
iii) the $P$ term in row 0, column $jk.$ 

Hence, 
\begin{equation*} 
\left. 
\begin{aligned} 
\; &\text{The sum of elements in column  $kj, \text{ for } 1 \le j  \le k-2,$} & \; & \;   \\
\;&=P_{k-1,k-j}+ \sum_{r=k-j}^{k-1} Q_{r,k-j} - P_{k-1,k-j-1}   \\
 \;&=2 P_{k-1,k-j} - P_{k-1,k-j-1}, & \qquad \text{by \eqref{Q2P}(a)},      \\
\;& =	P_{k,k-j}. &\; \text{by \eqref{2minus1}.} 
\end{aligned} 
\right\} 
\end{equation*}

\section{Proof of \eqref{SumMrc} for Case D, columns $jk+m, j=k-1, 0 \le m \le k-1$}
For the case $k=5$, Table \ref{Trqnumber25} motivates   proving
\eqref{SumMrc} using well-known identities about powers of 2.

For general $k,$ by \eqref{2minus1}, \eqref{Q}, and by the initial conditions
presented in Tables \ref{TP} and \ref{TQ}, a straightforward inductive argument shows that  
\begin{equation*}
	P_{n,1} = 2^n-1, \qquad	Q_{n,1} = 2^{n-1}, \qquad   n \ge 1.
\end{equation*}

By inspection of Table \ref{TrqQP25} for the case $k=5,$ and by \eqref{M},
for general $k,$ we see that column $jk+m, 0 \le m \le k-1$
has a 1 entry and elements  $Q_{1,1},\dotsc, Q_{k-1-m,1}.$ 
There are now two cases to consider according to the value of $m.$

\textbf{Case $m=0.$} If $m=0$ column $jk+m$ also has entry
$P_{k-1,1}.$  Hence, the following proves
\eqref{SumMrc} for column $jk+m, m = 0$  in the $(k-1)$-st block.   
\begin{equation*} 
	\sum_{r=0}^k M_{r,k(k-1)} =
	1 + \sum_{r=1}^{k-1} Q_{r,1} + P_{k-1,1} =
	1 + \sum_{r=1}^{k-1} 2^{r-1} + (2^{k-1}-1) =
	P_{k,1}. 
\end{equation*}

\textbf{Case $m=1.$} If $m \neq 0,$ column $jk+m$ also has entry $-Q_{k-m,1}.$
Hence, the following proves
\eqref{SumMrc} for column $jk+m, m \neq 0$  in the $(k-1)$-st block.
\begin{equation*} 
	\sum_{\substack{r=0\\m \neq 0}}^k M_{r,k(k-1)+m} =
	1 + \sum_{r=1}^{k-1-m} Q_{r,1} - Q_{k-m,1} =
	1 + \sum_{r=1}^{k-1-m} 2^{r-1} - 2^{k-m-1} = 
	0.
\end{equation*}

\section{Proof of \eqref{SumMrc} for Case E, \eqref{jkm}, \
columns $jk+m,$ for  $0 \le j \le k-2, 1 \le m \le j$ }

For each column in Case E whose sum we are trying to show equal to 0,
we will first show that the column entries naturally divide into 4 groups 
(with some groups empty). The proof is then completed by  
showing that the sum of the four group-sums equals 0.

\textbf{1st group:} 
As can be motivated by Tables \ref{TrqQP14} and \ref{TrqQP25}, 
and as justified for general $k$
by \eqref{M}, 
all columns satisfying \eqref{jkm} have a $-P$ term 
which descends 
diagonally from the column in that block divisible by $k.$ 

For example,  column 6 has a $-P_{4,3}$ term which diagonally descends from column 5;
columns 11 and 12 have a $-P_{4,2}$ term which diagonally descends from column 10; and columns 
16,17,18 each have a $-P_{4,1}$ term which diagonally descends from column 15.

Thus, for general $k,$ 
this first group of the elements in a column satisfying \eqref{jkm}, is
defined as the singleton $P$ term that diagonally descends from
the column of that block divisible by $k.$  
\begin{equation}\label{case1}
	\text{ Sum of the singleton element in 1st Group} = -P_{k-1,k-1-j}.
\end{equation}

\textbf{2nd group:} 
In the $X^5$ row of Tables \ref{TrqQP14} and \ref{TrqQP25}, 
any column satisfying \eqref{jkm} has  
a negative $Q$ terms that descends from  
column $(j-1)k$ in the $(j-1)$-st block. When $j=0,$ this group is vacuously
empty; for $j=1,$ column 6 has an entry $-Q_{4,4},$ 
which diagonally descends from column 5; 
for $j=2,$ columns 11 and 12 have entries of $-Q_{4,3},  -Q_{3,3}$ respectively,
which diagonally descend from column 10; 
and for $j=3,$ columns 16, 17, 18 have entries $-Q_{4,2}, -Q_{3,2}, -Q_{2,2}$ respectively, 
which diagonally descend from column 15.

By \eqref{M}, for general $k,$ this second group of elements in a column
satisfying \eqref{jkm}, is defined as the negative $Q$ term that diagonally
descends from column $(j-1)k$ in the $(j-1)$-st block. By \eqref{Q}, 
for each $m$ and $j$ satisfying \eqref{jkm},
 the sum of the singleton element in this group is the following.
\begin{equation}\label{case2}
	- Q_{k-m,k-j} = - \biggl(P_{k-m,k-j} - P_{k-m-1,k-j}  \biggr).
\end{equation}

\textbf{3rd group:}  For each $m$ satisfying \eqref{jkm},
we define this third group, as
the $Q$ entries in row 0 of the $j$-th block
diagonally descending to column $jk+m$ (this includes the $j$-th block entry
already in that column in row 0).
 
For example, Tables \ref{TrqQP14} and \ref{TrqQP25} show 
that column 6 has a singleton $Q_{4,3};$ column
11 has a singleton $Q_{4,2},$ and column 12 has $Q_{3,2}, Q_{4,2};$
column 16 has a singleton $Q_{4,1},$ column 17 has $Q_{3,1}, Q_{4,1}$
and column 18 has $Q_{2,1}, Q_{3,1}, Q_{4,1}.$ 
 
Using \eqref{M}, for general $k,$  
the set of $Q$ entries in row 0 of the $j$-th block
diagonally descending to column $jk+m, 0 \le m \le j$ 
equals $\{Q_{k-m,k-j-1},\dotsc, Q_{k-1,k-j-1}\}.$
By \eqref{Q2P}(b), the  sum of  elements in this group is as follows:

\begin{equation}\label{case3}
	\displaystyle \sum_{r=k-m}^{k-1} Q_{r,k-j-1} = P_{k-1,k-j-1} -P_{k-m-1,k-j-1}.
\end{equation}

\textbf{4th group:} For each fixed $j,k,m$ satisfying \eqref{jkm}, we define
this fourth group to consist of the  $Q$ entries diagonally descending
from row 0 in the $j-1$-st block.

This last group requires more care to address 
 since for some columns the group is empty.

For example,  Column 16 in Table \ref{TrqQP25} has $Q_{2,2}$ and $Q_{3,2}$ while 
column 17 has singleton $Q_{2,2}.$ Column 18 has no
entries for this group; but we can alternatively describe this empty group
as all $Q$s between $Q_{2,2}$ and $Q_{1,2}$
which would result in no $Q$s since there are no integers between 2 and 1.

For general $k,$ by\eqref{M}, this group consists of the 
following (possibly) empty set of elements: $\{Q_{k-j,k-j}, \dotsc, Q_{k-1-m,k-j}\}$
where we follow the usual convention that the set is empty if 
$k-1-m < k-j.$ 

Hence, by \eqref{Q2P}(a),
the sum of the entries in  this fourth group is the following.
\begin{equation}\label{case4}
	\displaystyle \sum_{r=k-j}^{k-1-m} Q_{r,k-j} = P_{k-m-1,k-j}.
\end{equation}
 
It is straightforward to check for Tables \ref{TrqQP14} and 
\ref{TrqQP25} that the sum is vacuous for 
$k=5$ at columns 6 and 18 and that the $P$ value on the right hand side of 
\eqref{case4} is also 0 as required.

To complete the proof for this case (E), 
we note that for any column $jk+m,$ satisfying \eqref{jkm}, 
i) the four groups enumerated above are mutually exclusive,
ii) each entry in column $jk+m$ is in some group, and 
iii) by \eqref{2minus1} the sum of
the right-hand sides of \eqref{case1}-\eqref{case4} is 0.

This completes the proof for case E. 

\section{Completion of Proof of Theorem 1.3}

To see that the proof of Theorem 1.3 is complete we again review the flow of logic.
\begin{itemize}
\item The last five sections have proven \eqref{SumMrc}, which for Tables 
\ref{TrqQP14} and \ref{TrqQP25} state that the \emph{sum} row of $M$ contains
the sum of all column entries above it, with the column entries being obtained by
the multiplication of $q_5(X),$ the top row, by $-r_5(X)$ the left-most column. 
\item Equation \eqref{SumMrc} in turn proves \eqref{qkn},
$-p_k(X^k) = -r_k(X) \times q_k(X).$  The long-hand multiplication
is illustrated for the case $k=5$ in Tables \ref{TrqQP14}-\ref{TrqQP25}. 
\item Equation \eqref{qkn} is equivalent to \eqref{qk}, $p_k(X^k) =  r_k(X)q_k(X)$
\item Equation \eqref{qk} directly implies that $r_k(X),$ the characteristic polynomial
for $\{F^{(k)}_{n}\}_{\text{all $n$}},$
divides   $p_k(X^k),$ the characteristic polynomial for 
$\{F^{(k)}_{kn}\}_{\text{all $n$}}.$
\item Hence by either Table \ref{Tpolynomial} or Proposition 1.4,
which creates a \emph{dictionary} between polynomial identities and recursions,
$F_n = \sum_{i=1}^k P_{k,i} F_{n-ki},$ the identity associated with the polynomial
$p_k(X^k),$  is satisfied by the sequence $\{F^{(k)}_{n}\}_{\text{all $n$}}.$
\item This last bullet is the statement of Theorem 1.3, which is therefore proven.
\end{itemize} 
 \section{Conclusion}

In this paper, we have proven a family of identities for the family of
recursive sequences $F^{(k)}, k=2,3,\dotsc$
restricted to a fix modulus modulo $k.$ This paper has provided a
 proof-method tool which   exploits the divisibility
properties of characteristic polynomials and was facilitated by a matrix $M$
showing the long-hand multiplication.  We believe this tool may prove
useful for proving families of identities in other infinite families of recursions.
A typical application of the methods of this paper could involve starting with
some second order recursion such as the Pell, Pell-Lucas, or Jacobstahl sequence,
generalizing this sequence to a family of recursive sequences similar to the
$F^{(k)},k=2,3,\dotsc,$ and uniformly proving identities satisfied by   the subsequences
whose indices are divisible by $k.$

\medskip

\noindent MSC2020: 11B39 

\end{document}